\documentclass[a4paper,12pt,reqno]{amsart}
\usepackage{amssymb,amsmath,amsfonts, amscd}
\usepackage{mathrsfs,latexsym}
\usepackage[all,ps,cmtip]{xy}
\usepackage[usenames,dvipsnames,svgnames,table]{xcolor}
\usepackage{extarrows} %%%%%%%%%%%%%%%%%%%%%%%
\title{On minimal 4-folds of general type with $p_g \geq 2$}
\author{Jianshi Yan}

\address{\rm School of Mathematical Sciences, Fudan University, Shanghai 200433, China}
\email{jsyan15@fudan.edu.cn}


\newcommand{\bQ}{{\mathbb Q}}
\newcommand{\bP}{{\mathbb P}}
\newcommand{\roundup}[1]{\lceil{#1}\rceil}

\newcommand\Vol{\text{\rm Vol}}
\newcommand\lrw{\longrightarrow}

\newcommand{\lsgeq}{\succcurlyeq}
\newcommand{\lsleq}{\preccurlyeq}

\newcommand{\Mov}{\text{Mov}}

\newtheorem{thm}{Theorem}[section]
\newtheorem{lem}[thm]{Lemma}

\newtheorem{prop}[thm]{Proposition}

\theoremstyle{definition}

\theoremstyle{remark}
\newtheorem{rem}[thm]{Remark}

\newtheorem{proof of thm}{\bf Proof of Theorem \ref{modified lemma for 3-fold}}
\begin{document}
\begin{abstract}
We show that, for nonsingular projective 4-folds V of general type with geometric genus $p_g\geq 2$, $\varphi_{33}$ is birational onto the image and the canonical volume $\Vol(V)$ has the lower bound $\frac{1}{520}$, which improves a previous theorem by Chen and Chen.
\end{abstract}
\maketitle

%%%%%%%%%
\pagestyle{myheadings}
\markboth{\hfill J. Yan\hfill}{\hfill On minimal 4-folds of general type with $p_g\geq 2$\hfill}
\numberwithin{equation}{section}
%%%%%%%%%%%%

%%%%%%%%%%%%%%%
\section{\bf Introduction}
%%%%%%%%%%%%%%%

Studying the behavior of pluricanonical maps of projective varieties has been one of the fundamental tasks in birational geometry. For varieties of general type, an interesting and critical problem is to find a positive integer $m$ so that $\varphi_m$ is birational onto the image. A momentous theorem given by Hacon-McKernan \cite{H-M}, Takayama \cite{Ta} and Tsuji \cite{Tsuji} says that there is some constant $r_n$ (for any integer $n>0$) such that the pluricanonical map $\varphi_m$ is birational onto its image for all $m \geq r_n$ and for all minimal projective $n$-folds of general type. By using the birationality principle (see, for example, Theorem \ref{bir. prin.}), an explicit upper bound of $r_{n+1}$ is determined by that of $r_n$. Therefore, finding the explicit constant $r_n$ for smaller $n$ is the upcoming problem. However, $r_n$ is known only for $n\leq 3$, namely, $r_1=3$, $r_2=5$ by Bombieri \cite{Bom} and $r_3\leq 57$ by Chen-Chen \cite{EXP1,EXP2, EXP3} and Chen \cite{Delta18}.

The first partial result concerning the explicit bound of $r_4$ was due to \cite[Theorem 1.11]{EXP3} by Chen and Chen that $\varphi_{35}$ is birational for all nonsingular projective 4-folds of general type with $p_g\geq 2$.  It is mysterious whether the numerical bound ``35'' is optimal under the same assumption.

In this paper, we go on studying this question and prove the following theorem:

\begin{thm}\label{main theorem}
Let $V$ be a nonsingular projective 4-fold of general type with $p_g(V) \geq 2$. Then
\begin{itemize}
\item[(1)] $\varphi_{m,V}$ is birational for all $m\geq 33$;
\item[(2)] $\Vol(V)\geq \frac{1}{520}$.
\end{itemize}
\end{thm}

\begin{rem} As being pointed out by Brown and Kasprzyk \cite{B-K}, the requirement on $p_g$ in Theorem \ref{main theorem}(1) is indispensable from the following list of canonical fourfolds:
\begin{enumerate}
\item $X_{78}\subset\bP(39,14,9, 8,6,1)$, $Vol(X_{78})=1/3024$;
\item $X_{78}\subset \bP(39,13,10,8,6,1)$, $Vol(X_{78})=1/3120$;
\item $X_{72}\subset \bP(36,11,9,8,6,1)$, $Vol(X_{72})=1/2376$;
\item $X_{70}\subset \bP(35,14,10,6,3,1)$, $Vol(X_{70})=1/1260$;
\item $X_{70}\subset \bP(35,14,10,5,4,1)$, $Vol(X_{70})=1/1400$;
\item $X_{68}\subset \bP(34, 12,8,7,5,1)$, $Vol(X_{68})=1/1680$.
\end{enumerate}

Besides, the following two hypersurfaces has $p_g \geq 2$ and $\varphi_{17}$ is non-birational, so one may expect that $18$ is the optimal lower bound of $m$ such that $\varphi_m$ is birational for a nonsingular projective 4-fold of general type with $p_g \geq 2$:
  \begin{itemize}
  \item[(1)] $X_{36}\subset \bP^5(18,6,5,4,1,1)$;
  \item[(2)] $X_{36}\subset \bP^5(18,7,5,3,1,1)$.
  \end{itemize}
\end{rem}

Throughout all varieties are defined over a field $k$ of characteristic $0$. We will frequently use the following symbols:
\begin{itemize}
\item[$\diamond$] `$\sim$' denotes linear equivalence or $\mathbb{Q}$-linear equivalence;
\item[$\diamond$] `$\equiv$' denotes numerical equivalence;
\item[$\diamond$] `$|A| \lsgeq |B|$' or, equivalently, `$|B| \lsleq |A|$'  means $|A| \supseteq |B|+$ fixed effective divisors.
\end{itemize}

%%%%%%%%%%%%%%%%%%%
\section{\bf Preliminaries}\label{pre}
%%%%%%%%%%%%%%%%%%%

Let $V$ be a nonsingular projective 4-fold of general type with geometric genus $p_g(V) \geq 2$. By the minimal model program (see, for instance \cite{BCHM, KMM,  K-M,  Siu}), we can find a minimal model $Y$ of $V$ with at worst $\mathbb{Q}$-factorial terminal singularities. Since the properties, which we study on $V$, are birationally invariant in the category of normal varieties with canonical singularities, we do focus our study on $Y$ instead. Clearly the canonical sheaf satisfies $\omega_Y\cong \mathcal{O}_Y(K_Y)$ for any canonical divisor $K_Y$.

\subsection{Convention}

\

For an arbitrary linear system $|D|$ of positive dimension on a normal projective variety $Z$, we may define {\it a generic irreducible element of $|D|$} in the following way. We have $|D|= \text{Mov}|D|+\text{Fix}|D|$, where $\text{Mov}|D|$ and $\text{Fix}|D|$ denote the moving part and the fixed part of $|D|$ respectively. Consider the rational map $\varphi_{|D|}=\varphi_{\text{Mov}|D|}$. We say that {\it $|D|$ is composed of a pencil} if $\dim \overline{\varphi_{|D|}(Z)}=1$; otherwise, {\it $|D|$ is not composed of a pencil}. {\it A generic irreducible element of $|D|$} is defined to be an irreducible component of a general member in $\text{Mov}|D|$ if $|D|$ is composed of a pencil or,  otherwise,  a general member of $\text{Mov}|D|$.

Keep the above settings. We say that $|D|$ can {\it distinguish different generic irreducible elements $X_1$ and $X_2$ of a linear system $|M|$ on $Z$} if neither $X_1$ nor $X_2$ is contained in $\text{Bs}|D|$, and if $\overline{\varphi_{|D|}(X_1)} \nsubseteq \overline{\varphi_{|D|}(X_2)}, \overline{\varphi_{|D|}(X_2)} \nsubseteq \overline{\varphi_{|D|}(X_1)}$.

A nonsingular projective surface $S$ of general type with $K_{S_0}^2=u$ and $p_g(S_0)=v$ is referred to as {\it a $(u,v)$-surface}, where $S_0$ is the minimal model of $S$.

\subsection{Set up for the map $\varphi_{1,Y}$}\label{setup}
\

Fix an effective divisor $K_1 \sim K_Y$. By Hironaka's theorem, we may take a series of blow-ups along nonsingular centers to obtain the model $\pi: Y' \rightarrow Y$ satisfying the following conditions:

(i) $Y'$ is nonsingular and projective;

(ii) the moving part of $|K_{Y'}|$ is base point free so that $$g_1=\varphi_{1,Y}\circ\pi:Y'\rightarrow \overline{\varphi_{1,Y}(Y)} \subseteq{\mathbb{P}^{p_g(Y)-1}}$$ is a non-trivial morphism;

(iii) the union of $\pi^*(K_1)$ and all those exceptional divisors of $\pi$ has simple normal crossing supports.

Take the Stein factorization of $g_1$.  We get
$$Y' \xlongrightarrow {f_1} \Gamma \xlongrightarrow s \overline{\varphi_{1,Y}(Y)},$$
and hence the following diagram commutes:
\smallskip

\begin{picture}(50,80) \put(100,0){$Y$} \put(100,60){$Y'$}
\put(180,0){$\overline{\varphi_{1,Y}(Y)}$} \put(200,60){$\Gamma$}
\put(115,65){\vector(1,0){80}}
\put(106,55){\vector(0,-1){41}} \put(203,55){\vector(0,-1){43}}
\put(114,58){\vector(2,-1){80}} \multiput(112,2.6)(5,0){13}{-}
\put(172,5){\vector(1,0){4}}
\put(145,70){$f_1$}
\put(208,30){$s$}
\put(92,30){$\pi$}
\put(132,-6){$\varphi_{1,Y}$}
\put(148,43){$g_1$}
\end{picture}
\bigskip

We may write
$$K_{Y'}=\pi^*(K_Y)+E_{\pi},$$
where $E_{\pi}$ is a sum of distinct exceptional divisors with positive rational coefficients. Denote by $|M_1|$ the moving part of $|K_{Y'}|$.
Since $Y$ has at worst $\mathbb{Q}$-factorial terminal singularities, we may write
$$\pi^*(K_Y) \sim M_1+E_1,$$
where $E_1$ is an effective $\mathbb{Q}$-divisor as well. %Clearly
One has $1\leq \dim(\Gamma) \leq 4$.

If $\dim(\Gamma)=1$, we have $M_1 \sim \sum\limits_{i=1}^{b} F_i\equiv bF$, where $F_i$ and $F$ are smooth fibers of $f_1$ and $b=\deg {f_1}_*\mathcal{O}_{Y'}(M_1) \geq p_g(Y)-1 \geq 1$. More specifically, when $g(\Gamma)=0$, we say that {\it $|M_1|$ is composed of a rational pencil} and when $g(\Gamma)>0$, we say that {\it $|M_1|$ is composed of an irrational pencil}.

If $\dim(\Gamma)>1$, by Bertini's theorem, we know that general members $T_i \in |M_1|$ are nonsingular and irreducible.

Denote by $T'$ a generic irreducible element of $|M_1|$.
Set
$$\theta_1=\theta_{1, |M_1|}=
\begin{cases}
b, & \text{if}\ \dim(\Gamma)=1;\\
1, &\text{if}\ \dim(\Gamma)\geq 2.
\end{cases}$$
So we naturally get
$$\pi^*(K_Y) \equiv \theta_1T'+E_1.$$

\subsection{Fixed notations}\label{notations}

\
%Y 4-fold, T 3-fold, S a surface, C a curve

Pick up a generic irreducible element $T'$ of $|M_1|$. Modulo further blow-ups on $Y'$, which is still denoted as $Y'$ for simplicity, we may have a birational morphism
$\pi_T=\pi|_{T'}: T'\rightarrow T$ onto a minimal model $T$ of $T'$. Let $t_1$ be the smallest positive integer such that $P_{t_1}(T) \geq 2$.

Set $|N|=\text{Mov}|t_1K_{T'}|$ and let $\varphi_{t_1, T}$ be the $t_1$-canonical map: $T \dashrightarrow \mathbb{P}^{P_{t_1}-1}$.
Similar to the 4-fold case, take the Stein factorization of the composition:
$$\varphi_{t_1, T} \circ \pi_T : T' \xlongrightarrow j \Gamma' \longrightarrow \overline{\varphi_{t_1, T}(T)}.$$
Denote by $j$ the induced projective morphism with connected fibers from $\varphi_{t_1, T} \circ \pi_T$ by Stein factorization.
Set
$$a_{t_1,T}=
\begin{cases}
c, & \text{if}\ \dim(\Gamma')=1;\\
1, &\text{if}\ \dim(\Gamma')\geq 2,
\end{cases}$$
where $c=\deg{j_*\mathcal{O}_{T'}(N)}\geq P_{t_1}(T)-1$.
Let $S$ be a generic irreducible element of $|N|$. Then we have
$$t_1\pi_T^*(K_T) \equiv a_{t_1,T}S+E_{N},$$
where $E_{N}$ is an effective $\mathbb{Q}$-divisor. Denote by $\sigma: S \rightarrow S_0$ the contraction morphism of $S$ onto its minimal model $S_0$.

Suppose that $|H|$ is a base point free linear system on $S$. Let $C$ be a generic irreducible element of $|H|$. As $\pi_T^*(K_T)|_S$ is nef and big, by Kodaira's lemma, there is a rational number $\tilde{\beta} >0$ such that $\pi_T^*(K_T)|_S \geq \tilde{\beta} C$.

Set
\begin{eqnarray*}
\beta=\beta(t_1,|N|,|H|)&=& \sup\{\tilde{\beta}|\tilde{\beta}>0 ~s.t.~ \pi_T^*(K_T)|_S\geq \tilde{\beta} C\} \\
\xi=\xi(t_1, |N|,|H|)&=&(\pi_T^*(K_T)\cdot C)_{T'}.
\end{eqnarray*}

\subsection{Technical preparation}

\

We will use the following theorem which is a special form of Kawamata's extension theorem (see \cite[Theorem A]{KawaE}).

\begin{thm}(cf. \cite[Theorem 2.2]{CHP}) \label{KaE}
Let $Z$ be a nonsingular projective variety on which $D$ is a smooth divisor.  Assume that $K_Z+D\sim A+B$ where $A$ is an ample $\bQ$-divisor and $B$ is an effective $\bQ$-divisor such that $D\not\subseteq \textrm{Supp}(B)$.
Then the natural homomorphism
$$H^0(Z, m(K_Z+D))\longrightarrow H^0(D, mK_D)$$
is surjective for any integer $m>1$.
\end{thm}

In particular, when $Z$ is of general type and $D$,  as a generic irreducible element, moves in a base point free linear system, the conditions of Theorem \ref{KaE} are automatically satisfied. Keep the settings as in \ref{setup} and \ref{notations}. Take $Z=Y'$ and $D=T'$. Then for sufficiently large and divisible integer $n>0$, we have
$$|n(\frac{1}{\theta_1}+1)K_{Y'}| \lsgeq |\frac{n}{\theta_1}(K_{Y'}+T')|$$
and the homomorphism
$$H^0(Y', \frac{n}{\theta_1}(K_{Y'}+T')) \rightarrow H^0(T', \frac{n}{\theta_1} K_{T'})$$
is surjective. By \cite[Theorem 3.3]{K-M}, $|\frac{n}{\theta_1} K_{T'}|$ is base point free, so one has $$\textrm{Mov}|\frac{n}{\theta_1} K_{T'}|=|\frac{n}{\theta_1}\pi_T^*(K_T)|.$$
It follows that
$$n(\frac{1}{\theta_1}+1)\pi^*(K_Y)|_{T'} \geq M_{n(\frac{1}{\theta_1}+1)}|_{T'} \geq  \frac{n}{\theta_1}\pi_T^*(K_T),$$
where the latter inequality holds by \cite[Lemma 2.7]{C01}.
So we get the {\it canonical restriction inequality}:
\begin{equation}
\pi^*(K_Y)|_{T'} \geq \frac{\theta_1}{1+\theta_1}\pi_T^*(K_T).\label{cri}
\end{equation}
Similarly, one has
\begin{equation}
\pi_T^*(K_T)|_S \geq \frac{a_{t_1,T}}{t_1+a_{t_1,T}}\sigma^*(K_{S_0}). \label{cri1}
\end{equation}

We will tacitly use the following type of birationality principle.

\begin{thm} (cf. \cite[2.7]{EXP2}) \label{bir. prin.}
Let $Z$ be a nonsingular projective variety, $A$ and $B$ be two divisors on $Z$ with $|A|$ being a base point free linear system. Take the Stein factorization of $\varphi_{|A|}$: $Z \overset{h} \longrightarrow W \longrightarrow \bP^{h^0(Z,A)-1}$ where $h$ is a fibration onto a normal variety $W$. Then the rational map $\varphi_{|B+A|}$ is birational onto its image if one of the following conditions is satisfied:
\begin{itemize}
\item [(i)] $\dim\varphi_{|A|}(Z)\geq 2$, $|B|\neq \emptyset$ and $\varphi_{|B+A|}|_D$ is birational for a
general member $D$ of $|A|$.

\item [(ii)] $\dim\varphi_{|A|}(Z)=1$, $\varphi_{|B+A|}$ can distinguish different general fibers of $h$ and $\varphi_{|B+A|}|_F$ is birational for a general fiber $F$ of $h$.
\end{itemize}
\end{thm}

\subsection{Some useful lemmas}\label{known results}

\

The following results on surfaces and 3-folds will be used in our proof.

\begin{lem}(see \cite[Lemma 2.5]{EXP3})\label{cfc}
Let $\sigma: S \lrw S_0$ be the birational contraction onto the minimal model $S_0$ from a nonsingular projective surface $S$ of general type.  Assume that $S$ is not a $(1,2)$-surface and that $\tilde{C}$ is a curve on $S$ passing through very general points. Then $(\sigma^*(K_{S_0})\cdot \tilde{C})\geq 2$.
\end{lem}

\begin{lem}(\cite[Lemma 2.5]{C14}) \label{curve lem.1}
Let $S$ be a nonsingular projective surface. Let $L$ be a nef and big $\bQ$-divisor on $S$ satisfying the following conditions:
\begin{itemize}
\item[(1)] $L^2>8$;
\item[(2)] $(L \cdot C_{x})\geq 4$ for all irreducible curves $C_{x}$ passing through any very general point $x\in S$.
\end{itemize}
Then $|K_S+\roundup{L}|$ gives a birational map.
\end{lem}

%%%%%%%%%%%%%
\section{\bf Proof of the main theorem}
%%%%%%%%%%%%%

As an overall discussion, we keep the same settings as in \ref{setup} and \ref{notations}.

\subsection{Separation properties of $\varphi_{m,Y}$}

\

%%%%%%%%%
%distinguish g.i.e.(3-fold)
\begin{lem}\label{dis. 3-fold} Let $Y$ be a minimal 4-fold of general type with $p_g(Y)\geq 2$. Then $|mK_{Y'}|$ can distinguish different generic irreducible elements of $|M_1|$ for all $m \geq 3$.
\end{lem}
\begin{proof} Suppose $m \geq 3$. As we have $mK_{Y'} \geq M_1$, by Theorem \ref{bir. prin.}, we may just consider the case when $|M_1|$ is composed of a pencil. In particular, when $|M_1|$ is composed of a rational pencil, which is the case when $\Gamma \cong \mathbb{P}^1$, the global sections of ${f_1}_*\mathcal{O}_{Y'}(M_1)$ can distinguish different points of $\Gamma$. So $|M_1|$, and consequently $|mK_{Y'}|$ can distinguish different general fibers of $f_1$. Hence we may just deal with the case when $|M_1|$ is composed of an irrational pencil. We have $M_1 \sim \sum\limits_{i=1}^{b} T_i$, where $T_i$ are smooth fibers of $f_1$ and $b \geq 2$. Pick two different generic irreducible elements $T_1$, $T_2$  of $|M_1|$. Then by Kawamata-Viehweg vanishing theorem (\cite{KV, VV}), one has
$$H^1(K_{Y'}+\roundup{(m-2)\pi^*(K_Y)}+M_1-T_1-T_2)=0,$$
and the surjective map
\begin{eqnarray}
&&H^0(Y', K_{Y'}+\roundup{(m-2)\pi^*(K_Y)}+M_1) \notag\\
&\lrw& H^0(T_1, \big(K_{Y'}+\roundup{(m-2)\pi^*(K_Y)}+M_1\big)|_{T_1})\label{e1}\\
&&\oplus  H^0(T_2, \big(K_{Y'}+\roundup{(m-2)\pi^*(K_Y)}+M_1\big)|_{T_2}).\label{e2}
\end{eqnarray}
Since $p_g(Y)\geq 2$, both $K_{T_i}$ and $\pi^*(K_Y)$ are effective. So for general $T_i$, $\pi^*(K_Y)|_{T_i}$ is effective. As $T_i$ is moving and $M_1|_{T_i}\sim 0$, both groups in (\ref{e1}) and (\ref{e2}) are non-zero. Therefore, $|mK_{Y'}|$ can distinguish different generic irreducible elements of $|M_1|$.
\end{proof}

%distinguish g.i.e.(surface)

\begin{lem}\label{ddgies} Let $Y$ be a minimal 4-fold of general type with $p_g(Y)\geq 2$. Pick up a generic irreducible element $T'$ of $|M_1|$. Then $|mK_{Y'}||_{T'}$ can distinguish different generic irreducible elements of $|N|$ for all $$m \geq 2t_1+4.$$
\end{lem}
\begin{proof}
Suppose $m \geq 2t_1+4$. We have $K_{Y'} \geq \pi^*(K_Y) \geq T'$. Similar to the proof of Lemma \ref{dis. 3-fold}, we consider the following two situations: (i) $|N|$ is not composed of a pencil or $|N|$ is composed of a rational pencil; (ii) $|N|$ is composed of an irrational pencil.

For (i), one has
$$|mK_{Y'}| \lsgeq |2(t_1+1)K_{Y'}| \lsgeq |(t_1+1)(K_{Y'}+T')|.$$
By Theorem \ref{KaE}, one has
$$|(t_1+1)(K_{Y'}+T')||_{T'} \lsgeq |(t_1+1)K_{T'}|.$$
As $(t_1+1)K_{T'} \geq N,$ $|mK_{Y'}||_{T'}$ can distinguish different generic irreducible elements of $|N|$.

For (ii), it holds that
\begin{eqnarray*} %\label{se.1}
|mK_{Y'}| \lsgeq |2(t_1+2)K_{Y'}| \lsgeq |(t_1+2)(K_{Y'}+T')|.
\end{eqnarray*}
Using Theorem \ref{KaE} again, one gets
\begin{eqnarray*}
& & |(t_1+2)(K_{Y'}+T')||_{T'}\\
&\lsgeq& |(t_1+2)K_{T'}|\\
&\lsgeq& |2K_{T'}+N|\\
&\lsgeq& |K_{T'}+\roundup{\pi_T^*(K_T)}+(N-S_1-S_2)+S_1+S_2|,
\end{eqnarray*}
where $S_1$ and $S_2$ are two different generic irreducible elements of $|N|$.
The vanishing theorem implies the surjective map
\begin{eqnarray}
&& H^0(T', K_{T'}+\roundup{\pi_T^*(K_T)}+N) \notag \\
&\rightarrow& H^0(S_1, (K_{T'}+\roundup{\pi_T^*(K_T)})|_{S_1}) \label{se1}\\
&& \oplus  H^0(S_2, (K_{T'}+\roundup{\pi_T^*(K_T)})|_{S_2}), \label{se2}
\end{eqnarray}
where we note that $(N-S_{i})|_{S_i}$ is linearly trivial for $i=1,2$.
Since $p_g(T)>0$,  both groups in \eqref{se1} and \eqref{se2} are non-zero. So $|mK_{Y'}||_{T'}$ can distinguish different generic irreducible elements of $|N|$ for any $m \geq 2t_1+4$.
\end{proof}

%distinguish g.i.e. (curves)

\begin{lem}\label{ddgiec} Let $Y$ be a minimal 4-fold of general type with $p_g(Y)\geq 2$. Pick up a generic irreducible element $T'$ of $|M_1|$ and a generic irreducible element $S$ of $|N|$. Define
$$|H|=\begin{cases} \text{Mov}|K_S|, & (K_{S_0}^2, p_g(S))=(1,2)\ \text{or}\ (2,3);\\
\text{Mov}|2K_S|, & \text{otherwise}.
\end{cases}$$
Then $|mK_{Y'}||_S$ can distinguish different generic irreducible elements of $|H|$ for all $m \geq 4(t_1+1).$
\end{lem}

\begin{proof}
Similar to the proof of Lemma \ref{ddgies}, we have
$$|mK_{Y'}||_{T'} \lsgeq |4(t_1+1)K_{Y'}||_{T'} \lsgeq |2(t_1+1)K_{T'}|.$$
Since $t_1K_{T'}\geq S$, we have
\begin{eqnarray*}
|2(t_1+1)K_{T'}||_S \lsgeq |2(K_{T'}+S)||_S \lsgeq |2K_S| \lsgeq |H|.
\end{eqnarray*}
As $p_g(S)>0$, $|H|$ is not composed of an irrational pencil, so the statement automatically follows.
\end{proof}

\subsection{Two useful propositions}\

\begin{prop}\label{birat.non12} Let $Y$ be a minimal 4-fold of general type with $p_g(Y)\geq 2$. Pick up a generic irreducible element $T'$ of $|M_1|$ and a generic irreducible element $S$ of $|N|$. If $S$ is not a $(1,2)$-surface, then $\varphi_{m,Y}$ is birational for all
$$m>(2\sqrt{2}+1)(\frac{t_1}{a_{t_1,T}}+1)(1+\frac{1}{\theta_1}).$$
\end{prop}
\begin{proof}
Suppose $m>(2\sqrt{2}+1)(\frac{t_1}{a_{t_1,T}}+1)(1+\frac{1}{\theta_1})$. Since
\begin{equation*}
(m-1)\pi^*(K_Y)-T'-\frac{1}{\theta_1}E_1
\equiv (m-1-\frac{1}{\theta_1})\pi^*(K_Y)
\end{equation*}
is nef and big, and it has simple normal crossing support, Kawamata-Viehweg vanishing theorem implies
\begin{eqnarray}
|mK_{Y'}||_{T'} &\lsgeq& |K_{Y'}+\roundup{(m-1)\pi^*(K_Y)-\frac{1}{\theta_1}E_1}||_{T'}\notag\\
&\lsgeq& |K_{T'}+\roundup{\big((m-1)\pi^*(K_Y)-T'-\frac{1}{\theta_1}E_1\big)|_{T'}}|\notag \\
&=& |K_{T'}+\roundup{Q_{m,T'}}|, \label{dd111}
\end{eqnarray}
where $Q_{m,T'}=((m-1)\pi^*(K_Y)-T'-\frac{1}{\theta_1}E_1)|_{T'} \equiv (m-1-\frac{1}{\theta_1})\pi^*(K_Y)|_{T'}$ is nef and big and has simple normal crossing support.

By the canonical restriction inequality \eqref{cri}, we may write
\begin{equation*}
\pi^*(K_Y)|_{T'}\equiv \frac{\theta_1}{1+\theta_1}\pi_T^*(K_T)+E_{1, T'},
\end{equation*}
where $E_{1, T'}$ is certain effective $\bQ$-divisor. As $t_1\pi_T^*(K_T)\equiv a_{t_1,T}S+E_{N}$, one may obtain
\begin{eqnarray}
|mK_{Y'}||_S
&\lsgeq& |K_{T'}+\roundup{Q_{m,T'}-\frac{1}{a_{t_1,T}}E_N}||_S\notag\\
&\lsgeq& |K_{S}+\roundup{(Q_{m,T'}-S-\frac{1}{a_{t_1,T}}E_N)|_S}|\notag\\
&\lsgeq& |K_S+\roundup{U_{m,S}}|, \label{ddS1}
\end{eqnarray}
where
\begin{eqnarray*}
U_{m,S}&=&(Q_{m,T'}-S-\frac{1}{a_{t_1,T}}E_N-(m-1-\frac{1}{\theta_1})E_{1,T'})|_S
 \\
&\equiv&((m-1-\frac{1}{\theta_1}) \cdot \frac{\theta_1}{1+\theta_1}-\frac{t_1}{a_{t_1,T}}) \pi_T^*(K_T)|_S.
\end{eqnarray*}

As \eqref{cri1} also gives
\begin{equation*}\pi^*_T(K_T)|_S \equiv\frac{a_{t_1,T}}{t_1+a_{t_1,T}}\sigma^*(K_{S_0})+
E_{t_1,S}\label{eqS1}
\end{equation*}
for some effective $\bQ$-divisor $E_{t_1,S}$ on $S$, together with \eqref{ddS1},  one has
\begin{eqnarray*}
U_{m,S}^2
&=&(((m-1-\frac{1}{\theta_1}) \cdot \frac{\theta_1}{1+\theta_1}-\frac{t_1}{a_{t_1,T}}) \pi_T^*(K_T)|_S)^2\\
&\geq& \big(((m-\frac{\theta_1+1}{\theta_1})\cdot \frac{\theta_1}{\theta_1+1}-
\frac{t_1}{a_{t_1,T}})\cdot \frac{a_{t_1,T}}{t_1+a_{t_1,T}}\big)^2\cdot K_{S_0}^2>8,
\end{eqnarray*}
where $U_{m,S}$ is nef and big.  Hence the statement clearly follows from Lemma \ref{cfc}, Lemma \ref{curve lem.1}, Lemma \ref{dis. 3-fold}, Lemma \ref{ddgies} and Theorem \ref{bir. prin.}.
\end{proof}

\begin{prop}\label{birat.non12non23} Let $Y$ be a minimal 4-fold of general type with $p_g(Y)\geq 2$. Pick up a generic irreducible element $T'$ of $|M_1|$ and a generic irreducible element $S$ of $|N|$. If $S$ is neither a $(1,2)$-surface nor a $(2,3)$-surface, then $\varphi_{m,Y}$ is birational for all
$$m \geq 6(t_1+1).$$
\end{prop}
\begin{proof}
Suppose $m \geq 6(t_1+1)$. Following the same procedures as in the proof of Lemma \ref{ddgies} and Lemma \ref{ddgiec}, one has
$$|mK_{Y'}| \lsgeq |3(t_1+1)(K_{Y'}+T')|$$
and
$$|mK_{Y'}||_{T'} \lsgeq |3(t_1+1)K_{T'}|.$$
Furthermore, one has
$$|mK_{Y'}||_S
\lsgeq |3(t_1+1)K_{T'}||_S
\lsgeq |3(K_{T'}+S)||_S = |3K_S|.$$
By virtue of Bombieri's result in \cite{Bom} that $|3K_S|$ gives a birational map unless $S$ is a $(1,2)$-surface or a $(2,3)$-surface, together with Lemma \ref{dis. 3-fold}, Lemma \ref{ddgies} and Theorem \ref{bir. prin.}, the statement holds.
\end{proof}

\subsection{The case of  $\dim(\Gamma) \geq 2$}\label{dim>1}
\ %%%%

We follow Chen-Chen's approach in \cite[Theorem 8.2]{EXP3} to deal with the case of  $\dim(\Gamma) \geq 2$.

\begin{thm}(\cite[Theorem 8.2]{EXP3}) \label{thm(dim2)}
Let $Y$ be a minimal 4-fold of general type with $p_g(Y)\geq 2$. Assume that $\dim(\Gamma) \geq 2$, then $\varphi_{m,Y}$ is birational for all $m\geq 15$.
\end{thm}

\begin{proof}
By Theorem \ref{bir. prin.}, we may just consider $\varphi_{m,Y'}|_{T'}$ for a general member $T' \in |M_1|$.
As we have $\theta_1=1$, \eqref{cri} gives
\begin{equation} \pi^*(K_Y)|_{T'} \geq \frac{1}{2}\pi_T^*(K_T). \label{EY}\end{equation}
Modulo some birational modifications, we may assume that $|M_1|_{T'}|$ is a base point free linear system.
Pick up a generic irreducible element $S$  of $|M_1|_{T'}|$.  It follows that
$$\pi^*(K_Y)|_{T'} \geq M_1|_{T'} \geq S.$$
Modulo $\mathbb{Q}$-linear equivalence, we have
\begin{equation}\label{beq}
K_{T'} \geq (\pi^*(K_Y)+T')|_{T'} \geq 2S.
\end{equation}
Using Theorem \ref{KaE}, we get
\begin{equation} \pi_T^*(K_T)|_S \geq \frac{2}{3} \sigma^*(K_{S_0}).\label{ET}\end{equation}
Thus, combining \eqref{EY} and \eqref{ET}, one gets
$$\pi^*(K_Y)|_S \geq \frac{1}{3}\sigma^*(K_{S_0}).$$
By \eqref{dd111}, we already have
$$|mK_{Y'}||_{T'} \lsgeq |K_{T'}+\roundup{Q_{m,T'}}|,$$
where $Q_{m,T'}=((m-1)\pi^*(K_Y)-T'-\frac{1}{\theta_1}E_1)|_{T'} \equiv (m-2)\pi^*(K_Y)|_{T'}$.
As $\pi^*(K_Y)|_{T'} \equiv S+E_S$ for some effective $\mathbb{Q}$-divisor $E_S$ on $T'$ and
$$Q_{m,T'}-S-E_S \equiv (m-3)\pi^*(K_Y)|_{T'}$$
is nef and big, Kawamata-Viehweg vanishing theorem implies
\begin{eqnarray*}
|mK_{Y'}||_S
&\lsgeq& |K_{T'}+\roundup{Q_{m,T'}-E_S}||_S \\
&\lsgeq& |K_S+\roundup{R_{m,S}'}|,
\end{eqnarray*}
where
\begin{eqnarray*}
R_{m,S}'&=& (Q_{m,T'}-S-E_S)|_S\\
&\equiv& (m-3)\pi^*(K_Y)|_S.
\end{eqnarray*}
Since $R_{m,S}' \equiv \frac{m-3}{3}\sigma^*(K_{S_0})+E_{m,S}'$, where $E_{m,S}'$ is an effective $\mathbb{Q}$-divisor on $S$, by Lemma \ref{curve lem.1},  $|K_S+\roundup{R_{m,S}'}|$ gives a birational map whenever $m \geq 15$.

Since $\text{Mov}|K_{T'}| \lsgeq |M_1|_{T'}|$, we may take $t_1=1$ and by the proof of Lemma \ref{ddgies} we know that $|mK_{Y'}|$ distinguishes different generic irreducible elements of $|M_1|_{T'}|$ for $m\geq 6$.
Therefore, $\varphi_{m,Y}$ is birational for all $m\geq 15$ in this case.
\end{proof}

\subsection{The case of $\dim(\Gamma)=1$} \label{dim=1}
\ %%%%

\begin{thm} \label{thm(dim1)}
Let $Y$ be a minimal 4-fold of general type with $p_g(Y) \geq 2$. Assume that $\dim(\Gamma)=1$, then $\varphi_{m,Y}$ is birational for all $m \geq 33$.
\end{thm}

\begin{proof} We have $\theta_1 \geq 1$ and $p_g(T')>0$. By Lemma \ref{dis. 3-fold}, $|mK_{Y'}|$ distinguishes different generic irreducible elements of $|M_1|$ for all $m \geq 3$.

As an overall discussion, we study the linear system $|mK_{Y'}||_C$ for generic irreducible element $C$ of $|H|$.  Recall that, by \eqref{dd111} and \eqref{ddS1}, we already have
$$|mK_{Y'}||_S\lsgeq |K_S+\roundup{U_{m,S}}|,$$
where $U_{m,S}\equiv \big((m-1-\frac{1}{\theta_1})\frac{\theta_1}{\theta_1+1}-\frac{t_1}{a_{t_1,T}}\big)\pi_T^*(K_T)|_S$ is a nef and big $\bQ$-divisor on $S$. As we have $\pi_T^*(K_T)|_S \sim \beta C+E_H$ for some effective $\bQ$-divisor $E_H$ on $S$, applying Kawamata-Viehweg vanishing theorem, we may get
\begin{eqnarray}
|mK_{Y'}||_C&\lsgeq&|K_S+\roundup{U_{m,S}-\frac{1}{\beta}E_H}||_C \notag \\
&=& |K_C+\roundup{U_{m,S}-C-\frac{1}{\beta}E_H}|_C| \notag \\
&=&|K_C+\mathcal{D}_m|, \label{rtc}
\end{eqnarray}
where $\mathcal{D}_m=\roundup{U_{m,S}-C-\frac{1}{\beta}E_H}|_C$ with
$$\deg \mathcal{D}_m\geq \big((m-1-\frac{1}{\theta_1})\frac{\theta_1}{\theta_1+1}-\frac{t_1}{a_{t_1,T}}-\frac{1}{\beta}\big)(\pi_T^*(K_T)|_S\cdot C).$$
Thus, whenever $m>\big(\frac{2}{\xi}+\frac{t_1}{a_{t_1,T}}+\frac{1}{\beta}+1\big)\cdot
\frac{\theta_1+1}{\theta_1}$, $|mK_{Y'}||_C$ gives a birational map.

Therefore, by Lemma \ref{ddgies}, Lemma \ref{ddgiec} and Theorem \ref{bir. prin.},
$\varphi_{m,Y}$ is birational provided that
\begin{equation*}\label{lb}
m \geq 4t_1+4  ~\text{and}~  m>\frac{4}{\xi}+\frac{2t_1}{a_{t_1,T}}+\frac{2}{\beta}+2.
\end{equation*}

Now we study this problem according to the value of $p_g(T)$.

\medskip

{\bf Case 1}.  $p_g(T) \geq 2$

Clearly, we may take $t_1=1$ and so $a_{t_1,T}=1$. From \cite[Section 2, Section 3]{MA}, we know that one of the cases occur:
\begin{enumerate}
\item[(1)] $\beta=1$, $\xi\geq 1$; (correspondingly, $d\geq 2$ or $d=1$ and $b>0$ in \cite{MA})
\item[(2)] $\beta=\frac{1}{4}$, $\xi\geq \frac{5}{4}$; (correspondingly, $d=1$ and $(1,1)$-surface case in \cite{MA})
\item[(3)] $\beta=\frac{1}{2}$, $\xi\geq \frac{2}{3}$; (correspondingly, $d=1$ and $(1,2)$-surface case in \cite{MA})
\item[(4)] $\beta=\frac{1}{2}$, $\xi\geq 1$; (correspondingly, $d=1$ and $(2,3)$-surface case in \cite{MA})
\item[(5)] $\beta=\frac{1}{4}$, $\xi\geq 2$. (correspondingly, $d=1$ and other surface case in \cite{MA})
\end{enumerate}

So $\varphi_{m,Y}$ is birational for all $m\geq 16$.
\medskip

{\bf Case 2}.  $p_g(T)=1$

According to \cite[Corollary 4.10]{EXP3}, $T$ must be of one of the following types: (i) $P_4(T)=1, P_5(T) \geq 3$; (ii) $P_4(T) \geq 2$.

For Type (i), we have $t_1=5$ and set $|N|=\Mov|5K_T|$.  When $|N|$ is composed of a pencil, we have $a_{t_1,T}\geq 2$ and $S$ is exactly the general fiber of the induced fibration from $\varphi_{t_1, T} \circ \pi_T$.  If $S$ is not a $(1,2)$-surface, by Proposition \ref{birat.non12}, $\varphi_{m,Y}$ is birational for all $m\geq 27$. If $S$ is a $(1,2)$-surface, by \cite[Proposition 4.1, Case 2]{CHP}, we have  $\beta\geq \frac{2}{7}$ and $\xi\geq \frac{2}{7}$, and hence $\varphi_{m,Y}$ is birational for $m\geq 28$.
When $|N|$ is not composed of a pencil, we have $a_{t_1,T}\geq 1$. Refer to the case by case argument of \cite[Proposition 4.2, Proposition 4.3]{CHP}, to give an exact list, $(\beta,\xi)$ must be among one of the situations:
$(1/5,3/7)$, $(1/5,2/3)$, $(1/3,1/3)$, $(1/5,5/13)$, $(1/5,1)$, $(1/2,1/3)$, $(1/5,1/2)$, $(2/5,1/3)$, $(1/4,1/3)$. Hence $\varphi_{m,Y}$ is birational for all $m\geq 33$.

For Type (ii), we have $t_1=4$ and set $|N|=\Mov|4K_T|$. When $|N|$ is composed of a pencil and the generic irreducible element $S$ of $|N|$ is neither a $(1,2)$-surface nor a $(2,3)$-surface, by Proposition \ref{birat.non12non23}, $\varphi_{m,Y}$ is birational for all $m\geq 30$.  When $P_4(T)=h^0(T,\mathcal{O}_T(4K_T))=2$ and $|N|$ is composed of a rational pencil of $(1,2)$-surfaces, the case by case argument of \cite[Proposition 4.6, Proposition 4.7]{CHP} tells that $(\beta,\xi)$ must be among one of the situations: $(2/7,2/7)$, $(1/5,2/5)$, $(2/5,2/7)$, $(1/5, 1/3)$, $(1/5, 2/3)$, $(1/5, 5/12)$, $(1/3, 2/7)$, $(1/4, 2/7)$.  Hence $\varphi_{m,Y}$ is birational for all $m\geq 33$. Otherwise, the case by case argument of \cite[Proposition 4.5]{CHP} tells that $(\beta,\xi)$ must be among one of the situations: $(1/4,2/5)$, $(1/5,2/5)$, $(1/3,1/3)$. Hence $\varphi_{m,Y}$ is birational for all $m\geq 31$.

In conclusion, $\varphi_{m, Y}$ is birational for all $m \geq 33$.
\end{proof}

\subsection{The canonical volume of 4-folds}

\

\begin{thm} \label{volume}
Let $Y$ be a minimal 4-fold of general type with $p_g(Y) \geq 2$. Then $\text{Vol}(Y) \geq \frac{1}{520}$.
\end{thm}

\begin{proof}
We have $\text{Vol}(Y)=K_Y^4=(\pi^*(K_Y))^4$.

Recall that $\pi^*(K_Y) \equiv \theta_1T'+E_1$. One has
$$\text{Vol}(Y)
\geq \theta_1(\pi^*(K_Y))^3 \cdot T'
= \theta_1(\pi^*(K_Y)|_{T'})^3.$$
As we also have \eqref{cri} and $t_1\pi_T^*(K_T) \equiv a_{t_1,T}S+E_{N}$,
it follows that
\begin{eqnarray*}
\text{Vol}(Y)
&\geq& \theta_1 \cdot (\frac{\theta_1}{1+\theta_1})^3(\pi_T^*(K_T))^3 \\
&\geq& \theta_1 \cdot (\frac{\theta_1}{1+\theta_1})^3 \cdot \frac{a_{t_1,T}}{t_1}(S \cdot (\pi_T^*(K_T))^2)\\
&=& \theta_1 \cdot (\frac{\theta_1}{1+\theta_1})^3 \cdot \frac{a_{t_1,T}}{t_1} (\pi_T^*(K_T)|_S)^2.
\end{eqnarray*}
By \eqref{cri1} and $\pi_T^*(K_T)|_S \geq \beta C$, we may get
\begin{eqnarray*}
\text{Vol}(Y)
&\geq& \theta_1 \cdot (\frac{\theta_1}{1+\theta_1})^3 \cdot \frac{a_{t_1,T}}{t_1} \cdot (\frac{a_{t_1,T}}{t_1+a_{t_1,T}})^2 K_{S_0}^2
\end{eqnarray*}
and
\begin{eqnarray*}
\text{Vol}(Y)
&\geq& \theta_1 \cdot (\frac{\theta_1}{1+\theta_1})^3 \cdot \frac{a_{t_1,T}}{t_1} \cdot \beta (\pi_T^*(K_T)|_S \cdot C)\\
&=& \theta_1 \cdot (\frac{\theta_1}{1+\theta_1})^3 \cdot \frac{a_{t_1,T}}{t_1} \cdot \beta \xi.
\end{eqnarray*}

Now we estimate the canonical volume  according to the same classification of $T$ and $S$ as in Subsection \ref{dim>1} and Subsection \ref{dim=1}.
\medskip

(I) The case of $\dim(\Gamma) \geq 2$

Remember that in this case, $\theta_1=1, t_1=1, a_{t_1,T}=2$ (by \eqref{beq}) and $\pi_T^*(K_T)|_S \geq \frac{2}{3} \sigma^*(K_{S_0})$ (by \eqref{ET}). So we have
$\text{Vol}(Y) \geq \frac{1}{9}.$
\medskip

(II) The case of $\dim(\Gamma)=1$

Subcase (II-1). $p_g(T) \geq 2$.

As in Theorem \ref{thm(dim1)}, Case 1, $t_1=1, a_{t_1,T}=1$, so we correspondingly have the estimation as follows:
\begin{enumerate}
\item[(1)] $\beta=1$, $\xi\geq 1$, then $\text{Vol}(Y) \geq \frac{1}{8}$;
\item[(2)] $\beta=\frac{1}{4}$, $\xi\geq \frac{5}{4}$, then $\text{Vol}(Y) \geq \frac{5}{128}$;
\item[(3)] $\beta=\frac{1}{2}$, $\xi\geq \frac{2}{3}$, then $\text{Vol}(Y) \geq \frac{1}{24}$;
\item[(4)] $\beta=\frac{1}{2}$, $\xi\geq 1$, then $\text{Vol}(Y) \geq \frac{1}{16}$;
\item[(5)] $\beta=\frac{1}{4}$, $\xi\geq 2$, then $\text{Vol}(Y) \geq \frac{1}{16}$.
\end{enumerate}
\medskip

Subcase (II-2). $p_g(T)=1$.

We follow the same classification of $T$ as in Theorem \ref{thm(dim1)},  Case 2.

Recall that for Type (i), we have $t_1=5$.  When $|N|$ is composed of a pencil and the general fiber $S$ of the induced fibration from $\varphi_{t_1, T} \circ \pi_T$ is not a $(1,2)$-surface, we have $a_{t_1,T}\geq 2$, $\beta \geq \frac{1}{7}$, $\xi \geq (\frac{2}{7}\sigma^*(K_{S_0}) \cdot C) \geq \frac{4}{7}$, and thus $\text{Vol}(Y) \geq \frac{1}{245}$. When $|N|$ is composed of a pencil and the general fiber $S$ is a $(1,2)$-surface, we have $a_{t_1,T}\geq 2$, $\beta\geq \frac{2}{7}$, $\xi\geq \frac{2}{7}$, and hence $\text{Vol}(Y) \geq \frac{1}{245}$.
When $|N|$ is not composed of a pencil, we have $a_{t_1,T}\geq 1$. The corresponding lower bounds of $\text{Vol}(Y)$ to those of $(\beta,\xi)$ are as follows:
$\frac{3}{1400}, \frac{1}{300}, \frac{1}{360}, \frac{1}{520},\frac{1}{200}, \frac{1}{240}, \frac{1}{400}, \frac{1}{300}, \frac{1}{480}$.

For Type (ii), we have $t_1=4$. When $|N|$ is not composed of a pencil, then $\beta \geq \frac{1}{4}, \xi \geq \frac{2}{5}$ and $\text{Vol}(Y) \geq \frac{1}{320}$. When $|N|$ is composed of a pencil of $(2,3)$-surfaces, then $\beta \geq \frac{1}{5}, \xi \geq \frac{2}{5}$ and $\text{Vol}(Y) \geq \frac{1}{400}$. When $|N|$ is composed of a pencil of surfaces with $K_{S_0}^2 \geq 2$, then $\text{Vol}(Y) \geq \frac{1}{400}$.
When $P_4(T)=h^0(T,\mathcal{O}_T(4K_T))=2$ and $|N|$ is composed of a rational pencil of $(1,2)$-surfaces, the corresponding lower bounds of $\text{Vol}(Y)$ to those of $(\beta,\xi)$ are as follows: $\frac{1}{392}, \frac{1}{400}, \frac{1}{280}, \frac{1}{480}, \frac{1}{240}, \frac{1}{384}, \frac{1}{336},\frac{1}{448}.$ Otherwise, $\beta \geq \frac{1}{3}, \xi \geq \frac{1}{3}$ and $\text{Vol}(Y) \geq \frac{1}{288}$.

So we have shown $\text{Vol}(Y) \geq \frac{1}{520}$.
\end{proof}

\subsection{Proof of Theorem \ref{main theorem}}
\

\begin{proof}
Theorem \ref{thm(dim2)}, Theorem \ref{thm(dim1)} and Theorem \ref{volume} directly implies Theorem \ref{main theorem}.
\end{proof}

Acknowledgment. The author would like to express her gratitude to Meng Chen for his guidance over this paper and his encouragement to the author. The author would also like to thank Dr. Yong Hu for pointing out the nonexistence of a kind of fibration in Subcase(II-2) in the previous version, which improves my result in the previous version.

%%%%%%%% References %%%%%%%%%%

\end{document}